\documentclass[11pt]{amsart}
\usepackage{amssymb}
\usepackage{graphicx}                            
\usepackage{hyperref}
\usepackage{setspace}                            

\newtheorem{lemma}{Lemma}[section]
\newtheorem{theorem}{Theorem}[section]

\newcommand{\adfmod}[1]{~(\mathrm{mod}~#1)}
\newcommand{\adfhide}[1]{}
\newcommand{\ADFvfyParStart}[1]{}

\begin{document}
\title{A construction for regular-graph designs}
\author{A. D. Forbes}
\address{LSBU Business School,
London South Bank University,
103 Borough Road,
London SE1 0AA, UK.}
\email{anthony.d.forbes@gmail.com}
\author{C. G. Rutherford}
\address{LSBU Business School,
London South Bank University,
103 Borough Road,
London SE1 0AA, UK.}
\email{c.g.rutherford@lsbu.ac.uk}
%
%

\begin{abstract}
A regular-graph design is
a block design for which a pair $\{a,b\}$ of distinct points occurs in $\lambda+1$ or $\lambda$
blocks depending on whether $\{a,b\}$ is or is not an edge of a given $\delta$-regular graph.
Our paper describes a specific construction for regular-graph designs with
$\lambda = 1$ and block size $\delta + 1$.
We show that for $\delta \in \{2,3\}$, certain necessary conditions for the existence of such a design with $n$ points
are sufficient, with two exceptions in each case and two possible exceptions when $\delta = 3$.
We also construct designs of orders 105 and 117 for connected 4-regular graphs.
\end{abstract}

\maketitle

\section{Introduction}\label{sec:Introduction}

All graphs in this paper are simple
(undirected, with no loops and no multiple edges).
Let $G$ be a graph.
If $i$ is a vertex, the {\em neighbourhood} of $i$, denoted by $N(i)$, is the set of vertices $j$ for which $\{i,j\}$ is an edge.
We also define $N[i] = \{i\} \cup N(i)$, the {\em closed neighbourhood} of $i$.
The cardinality of $N(i)$ is called the {\em degree} of $i$.
If $|N(i)|$ is independent of $i$, we say that the graph is {\em regular}, or, more specifically,
$\delta$-{\em regular} if $|N(i)| = \delta$ for all vertices $i$.
We adopt the standard notation $K_n$ for the complete graph on $n$ vertices and
$C_n$ for the cycle graph on $n$ vertices.
If $G$ and $H$ are graphs, $G + H$ is their vertex-disjoint union.
The {\em girth} of a graph is the length of the shortest cycle in it.


Suppose $G$ is a $\delta$-regular graph with vertex set $V$ of cardinality $n$.
A {\em regular-graph design} for $G$ with point set $V$ and parameters $k$, $r$, $\lambda$
is a collection $\mathcal{B}$ of {\em blocks}, subsets of $V$, such that:
\begin{enumerate}
\item[(i)]Each block contains the same number of points, $k$;
\item[(ii)]Each point occurs in the same number of blocks, $r$;
\item[(iii)]Each pair of points that is not an edge of $G$ occurs in precisely $\lambda$ blocks;
\item[(iv)]Each pair of points that is an edge of $G$ occurs in precisely $\lambda + 1$ blocks.
\end{enumerate}
This definition is the same as that which features in the papers of
Bailey \& Cameron, \cite{BaileyCameron2009, BaileyCameron2011}, and Wallis, \cite{Wallis1996}.
We often refer to a regular-graph design simply as a design,
or, to emphasise the graph $G$, a design for $G$,
or, to emphasise only the cardinality $n$ of the graph's vertex set, a design of order $n$.

In the remainder of the paper we construct regular-graph designs by a specific method.
We consider only the case $\lambda = 1$, i.e.\ non-adjacent pairs occur once, adjacent pairs occur twice.
In Section~\ref{sec:The-basic-construction} we describe our basic construction.
We also employ recursive constructions involving group divisible designs, which are defined in
Section~\ref{sec:group-divisible-designs}.
In Sections~\ref{sec:2-regular-graphs} -- \ref{sec:4-or-more-regular-graphs}
we address the construction of regular-graph designs from $\delta$-regular graphs with $\delta \in \{2,3,4\}$.

\section{The basic construction}\label{sec:The-basic-construction}

%
%
%
%
%

Given a $\delta$-regular graph $G$ with $n$ vertices and girth at least 5,
we construct a design consisting of blocks of size $k = \delta + 1$, where
each edge of $G$ occurs in precisely two blocks, and each pair of non-adjacent vertices of $G$ occurs in precisely one block.
The number of pairs is
$$\dfrac{n(n - 1)}{2} + \dfrac{n \delta}{2},$$
the number of blocks is
$$b = \dfrac{n(n + \delta - 1)}{\delta (\delta + 1)},$$
and the number of blocks that contain a specific point is
$$r = \dfrac{bk}{n} = \dfrac{n - 1}{\delta} + 1.$$

In our construction the duplicated pairs arise from the set $\mathcal{B}_\mathrm{N}$ of blocks of the form
$N[i] = \{i\} \cup N(i)$, where $i$ is a vertex of $G$, and $N(i)$ is the set of neighbours of $i$.
Let us call these {\em neighbourhood blocks}.
For the construction to work, two things must be right.
\begin{enumerate}
\item[(i)]The neighbourhoods $N(i)$ must be edgeless sets of vertices. \\
Otherwise, if $a, b \in N(i)$ and $\{a,b\}$ is an edge, then $\{a,b\}$ appears in at least three blocks,
$N[i]$, $N[a]$ and $N[b]$.
\item[(ii)]We must have $|N(i) \cap N(j)| \le 1$ for distinct $i$, $j$. \\
Otherwise, if $\{a,b\} \subseteq N(i) \cap N(j)$ for some distinct $i$, $j$, $a$ and $b$,
then pair $\{i,j\}$ occurs in blocks $N[a]$ and $N[b]$;
consequently $\{i,j\}$ is an edge and so $\{i,j\}$ is in $N[i]$ and $N[j]$ as well.
\end{enumerate}
Both conditions are met if and only if the graph has girth at least 5.

If $n = \delta^2 + 1$, the design is complete; all pairs of points are accounted for
with their correct multiplicities in $\mathcal{B}_\mathrm{N}$.
This can happen only when $\delta \in \{0, 1, 2, 3, 7, 57\}$ in which case the graphs are
$K_1, K_2$, the cycle graph $C_5$, the Petersen graph, the Hoffman--Singleton graph and
a 57-regular 3250-vertex graph with girth 5 that might or might not exist,
\cite{HoffmanSingleton1960}, \cite[\S 11.12]{Cameron1994}.
If $n > \delta^2 + 1$, we need more blocks.
We call these {\em remainder blocks} and we denote them by $\mathcal{B}_\mathrm{R}$.
They contain all point pairs that are not in blocks of $\mathcal{B}_\mathrm{N}$.
Hence the block set of the design is
$$\mathcal{B} = \mathcal{B}_\mathrm{N} \cup \mathcal{B}_\mathrm{R},$$
and
$$|\mathcal{B}_\mathrm{N}| = n,~~~~ |\mathcal{B}_\mathrm{R}| = \dfrac{n (n - \delta^2 - 1)}{\delta (\delta + 1)}.$$
Obviously $|\mathcal{B}_\mathrm{R}|$ must be a non-negative integer, which is zero when $n = \delta^2 + 1$.
Also observe that the graph is encoded entirely in $\mathcal{B}_\mathrm{N}$.
Thus we arrive at the following admissibility conditions for design order $n$:
\begin{equation} \label{eqn:admissibility-conditions}
\begin{array}{rcl}
n & \ge & \delta^2 + 1, \\
n & \equiv & 1 \adfmod{\delta}, \\
n(n + \delta - 1) & \equiv & 0 \adfmod{\delta(\delta + 1)}.
\end{array}
\end{equation}
%
\begin{lemma} \label{lem-n-delta2-1}
Each point in a design of order $n$ for a $\delta$-regular graph is paired with
$n - \delta^2 - 1$ points in the design's remainder blocks.
\end{lemma}
%
\begin{proof}
Consider a vertex $i$ and suppose $N(i) = \{i_1, i_2, \dots, i_\delta\}$.
Then $i$ occurs in
$(n - 1)/\delta + 1$
blocks of which $\delta + 1$ are in $\mathcal{B}_\mathrm{N}$, namely
$N[i]$, $N[i_1]$, $N[i_2]$, \dots, $N[i_{\delta}]$.
Hence $i$ is paired with $\delta + \delta(\delta - 1) = \delta^2$ distinct points in blocks of $\mathcal{B}_\mathrm{N}$ and
with $n - 1 - \delta^2$ points in blocks of $\mathcal{B}_\mathrm{R}$.
\end{proof}

Before we end this section it is appropriate to state that this paper was motivated by
Peter Cameron's blog of 17 June 2023, \cite{Cameron2023}.

\begin{center}
\begin{minipage}{0.9\textwidth}\small
\dots\ This brought to my mind something I have discussed and  asked many people about: finding all Sylvester designs.
These are block designs with 36 points and 48 blocks of size 6,
two points lying in two blocks if they are adjacent in the
distance-transitive Sylvester graph, and in a unique block
otherwise. \dots
\end{minipage}
\end{center}
Although we do not solve Peter's query we can neatly illustrate the method described above by constructing a Sylvester design.
The Sylvester graph is 5-regular and has girth 5.
From the edges
{$\{0, 1\}$, $\{0, 3\}$, $\{0, 11\}$, $\{0, 19\}$, $\{0, 27\}$, $\{1, 4\}$, $\{1, 12\}$, $\{1, 20\}$, $\{1, 28\}$, $\{2, 4\}$, $\{2, 5\}$, $\{2, 14\}$, $\{2, 22\}$, $\{2, 30\}$, $\{3, 5\}$, $\{3, 6\}$, $\{3, 23\}$, $\{3, 31\}$, $\{4, 7\}$, $\{4, 15\}$, $\{4, 32\}$, $\{5, 8\}$, $\{5, 16\}$, $\{5, 24\}$, $\{6, 7\}$, $\{6, 12\}$, $\{6, 21\}$, $\{6, 34\}$, $\{7, 8\}$, $\{7, 13\}$, $\{7, 26\}$, $\{8, 9\}$, $\{8, 18\}$, $\{8, 27\}$, $\{9, 10\}$, $\{9, 19\}$, $\{9, 28\}$, $\{9, 35\}$, $\{10, 12\}$, $\{10, 13\}$, $\{10, 16\}$, $\{10, 31\}$, $\{11, 13\}$, $\{11, 17\}$, $\{11, 24\}$, $\{11, 32\}$, $\{12, 14\}$, $\{12, 25\}$, $\{13, 22\}$, $\{13, 33\}$, $\{14, 17\}$, $\{14, 27\}$, $\{14, 35\}$, $\{15, 16\}$, $\{15, 19\}$, $\{15, 29\}$, $\{15, 34\}$, $\{16, 17\}$, $\{16, 20\}$, $\{17, 21\}$, $\{17, 26\}$, $\{18, 20\}$, $\{18, 21\}$, $\{18, 25\}$, $\{18, 32\}$, $\{19, 21\}$, $\{19, 22\}$, $\{20, 23\}$, $\{20, 33\}$, $\{21, 30\}$, $\{22, 23\}$, $\{22, 25\}$, $\{23, 26\}$, $\{23, 35\}$, $\{24, 25\}$, $\{24, 28\}$, $\{24, 34\}$, $\{25, 29\}$, $\{26, 28\}$, $\{26, 29\}$, $\{27, 29\}$, $\{27, 33\}$, $\{28, 30\}$, $\{29, 31\}$, $\{30, 31\}$, $\{30, 33\}$, $\{31, 32\}$, $\{32, 35\}$, $\{33, 34\}$, $\{34, 35\}$}
we immediately obtain 36 neighbourhood blocks
$$ \mathcal{B}_\mathrm{N} = \{N[i]: i = 0, 1, \dots, 35\}.$$
That leaves 180 pairs unaccounted for,
{$\{0, 2\}$, $\{0, 7\}$, $\{0, 10\}$, $\{0, 16\}$, $\{0, 18\}$, $\{0, 25\}$, $\{0, 26\}$, $\{0, 30\}$, $\{0, 34\}$, $\{0, 35\}$, $\{1, 5\}$, $\{1, 8\}$, $\{1, 13\}$, $\{1, 17\}$, $\{1, 21\}$, $\{1, 22\}$, $\{1, 29\}$, $\{1, 31\}$, $\{1, 34\}$, $\{1, 35\}$, $\{2, 6\}$, $\{2, 9\}$, $\{2, 10\}$, $\{2, 11\}$, $\{2, 18\}$, $\{2, 20\}$, $\{2, 26\}$, $\{2, 29\}$, $\{2, 34\}$, $\{3, 4\}$, $\{3, 9\}$, $\{3, 13\}$, $\{3, 14\}$, $\{3, 15\}$, $\{3, 17\}$, $\{3, 18\}$, $\{3, 25\}$, $\{3, 28\}$, $\{3, 33\}$, $\{4, 9\}$, $\{4, 10\}$, $\{4, 17\}$, $\{4, 21\}$, $\{4, 23\}$, $\{4, 24\}$, $\{4, 25\}$, $\{4, 27\}$, $\{4, 33\}$, $\{5, 12\}$, $\{5, 13\}$, $\{5, 19\}$, $\{5, 21\}$, $\{5, 26\}$, $\{5, 29\}$, $\{5, 32\}$, $\{5, 33\}$, $\{5, 35\}$, $\{6, 9\}$, $\{6, 11\}$, $\{6, 16\}$, $\{6, 20\}$, $\{6, 22\}$, $\{6, 27\}$, $\{6, 28\}$, $\{6, 29\}$, $\{6, 32\}$, $\{7, 14\}$, $\{7, 16\}$, $\{7, 19\}$, $\{7, 20\}$, $\{7, 24\}$, $\{7, 25\}$, $\{7, 30\}$, $\{7, 31\}$, $\{7, 35\}$, $\{8, 11\}$, $\{8, 12\}$, $\{8, 15\}$, $\{8, 17\}$, $\{8, 22\}$, $\{8, 23\}$, $\{8, 30\}$, $\{8, 31\}$, $\{8, 34\}$, $\{9, 11\}$, $\{9, 17\}$, $\{9, 20\}$, $\{9, 25\}$, $\{9, 29\}$, $\{9, 33\}$, $\{10, 18\}$, $\{10, 21\}$, $\{10, 23\}$, $\{10, 24\}$, $\{10, 26\}$, $\{10, 27\}$, $\{10, 34\}$, $\{11, 12\}$, $\{11, 15\}$, $\{11, 20\}$, $\{11, 23\}$, $\{11, 29\}$, $\{11, 30\}$, $\{12, 15\}$, $\{12, 19\}$, $\{12, 23\}$, $\{12, 26\}$, $\{12, 30\}$, $\{12, 32\}$, $\{12, 33\}$, $\{13, 14\}$, $\{13, 15\}$, $\{13, 18\}$, $\{13, 21\}$, $\{13, 28\}$, $\{13, 29\}$, $\{13, 35\}$, $\{14, 15\}$, $\{14, 18\}$, $\{14, 19\}$, $\{14, 20\}$, $\{14, 24\}$, $\{14, 28\}$, $\{14, 31\}$, $\{15, 18\}$, $\{15, 23\}$, $\{15, 28\}$, $\{15, 30\}$, $\{16, 22\}$, $\{16, 25\}$, $\{16, 27\}$, $\{16, 28\}$, $\{16, 30\}$, $\{16, 32\}$, $\{16, 35\}$, $\{17, 22\}$, $\{17, 25\}$, $\{17, 31\}$, $\{17, 33\}$, $\{17, 34\}$, $\{18, 26\}$, $\{18, 28\}$, $\{18, 34\}$, $\{19, 20\}$, $\{19, 24\}$, $\{19, 26\}$, $\{19, 31\}$, $\{19, 32\}$, $\{19, 33\}$, $\{20, 24\}$, $\{20, 29\}$, $\{20, 31\}$, $\{21, 23\}$, $\{21, 24\}$, $\{21, 27\}$, $\{21, 29\}$, $\{21, 35\}$, $\{22, 27\}$, $\{22, 28\}$, $\{22, 31\}$, $\{22, 32\}$, $\{22, 34\}$, $\{23, 24\}$, $\{23, 27\}$, $\{23, 30\}$, $\{24, 27\}$, $\{24, 31\}$, $\{25, 30\}$, $\{25, 33\}$, $\{25, 35\}$, $\{26, 32\}$, $\{26, 33\}$, $\{26, 34\}$, $\{27, 28\}$, $\{27, 32\}$, $\{28, 32\}$, $\{29, 35\}$, $\{30, 35\}$, $\{31, 34\}$, $\{32, 33\}$},
from which it is easy to construct by hand the 12 remainder blocks,
\begin{center}
\begin{tabular}{l@{~}l}
 $\mathcal{B}_\mathrm{R} = $
 & $\{0, 2, 10, 18, 26, 34\}$, $\{0, 7, 16, 25, 30, 35\}$, $\{1, 5, 13, 21, 29, 35\}$, \\
 & $\{1, 8, 17, 22, 31, 34\}$, $\{2, 6, 9, 11, 20, 29\}$, $\{3, 4, 9, 17, 25, 33\}$, \\
 & $\{3, 13, 14, 15, 18, 28\}$, $\{4, 10, 21, 23, 24, 27\}$, $\{5, 12, 19, 26, 32, 33\}$, \\
 & $\{6, 16, 22, 27, 28, 32\}$, $\{7, 14, 19, 20, 24, 31\}$, $\{8, 11, 12, 15, 23, 30\}$,
\end{tabular}
\end{center}
to complete the design---forty-eight blocks of size 6.
See also \cite{BaileyCameronSoicherWilliams2020}.

\section{Group divisible designs}\label{sec:group-divisible-designs}
For the purpose of this paper, a {\em group divisible design}, $k$-GDD, of type $g_1^{u_1} g_2^{u_2} \dots g_r^{u_r}$
is an ordered triple ($V, \mathcal{G}, \mathcal{B}$)
where
\begin{enumerate}
\item[(i)]$V$ is a set of $u_1 g_1 + u_2 g_2 + \dots + u_r g_r$ {\em points},
\item[(ii)]$\mathcal{G}$ is a partition of $V$ into $u_i$ subsets of size $g_i$, $i = 1, 2, \dots, r$, called \textit{groups}, and
\item[(iii)]$\mathcal{B}$ is a collection of subsets of cardinality $k$, called \textit{blocks}, which has the property that each pair of points from distinct groups occurs in precisely one block but
    a pair of distinct points from the same group does not occur in any block.
\end{enumerate}
Our next lemma asserts the existence of the group divisible designs that we require elsewhere in the paper.
%
\begin{lemma} \label{lem-GDD-existence}
There exists
\begin{enumerate}
\item[(i)]a $3$-$\mathrm{GDD}$ of type $g^3$ whenever $g \ge 1$;
\item[(ii)]a $3$-$\mathrm{GDD}$ of type $5^{6t} m^1$ whenever $t \ge 1$, $m \equiv 1 \adfmod{2}$ and $m \le 5(6t - 1)$;
\item[(iii)]a $3$-$\mathrm{GDD}$ of type $5^8 17^1$;
\item[(iv)]a $4$-$\mathrm{GDD}$ of type $10^{3t} m^1$ whenever $t \ge 2$, $m \equiv 1 \adfmod{3}$ and $m \le 5(3t - 1)$;
\item[(v)]a $4$-$\mathrm{GDD}$ of type $10^4$;
\item[(vi)]
a $5$-$\mathrm{GDD}$ of type $105^{4t+1}$ for $t \ge 1$,
a $5$-$\mathrm{GDD}$ of type $117^{20t+1}$ for $t \ge 1$ and
a $5$-$\mathrm{GDD}$ of type $117^{20t+5}$ for $t \ge 0$.
\end{enumerate}
\end{lemma}
%
\begin{proof}
For (i)--(iii), see \cite{Zhu1984} and \cite{ColbournHoffmanRees1992}, or \cite[Theorems IV.4.1 and IV.4.2]{Ge2007}.
For (iv), see \cite[Theorem 1.2 (iii)]{ForbesForbes2018}.
For (v), see \cite[Theorem IV.4.6]{Ge2007}.
For (vi), see \cite{GeLing2005} or \cite[Theorem IV.4.16]{Ge2007}.
\end{proof}

\section{2-regular graphs}\label{sec:2-regular-graphs}

Before discussing 2-regular graphs we dismiss two trivial cases.
\begin{enumerate}
\item[(i)]For {$\delta = 0$}, the block size is 1, the only design is the empty set and the corresponding graph is $K_1$.
\item[(ii)]For {$\delta = 1$}, the block size is 2, and we have $n \ge 2$, $n \equiv 0 \adfmod{2}$.
The blocks of the design are the edges of the corresponding graph, which is $n/2$ copies of $K_2$,
together with the edges of $K_n$ on the same vertex set.
\end{enumerate}

For {$\delta = 2$}, the block size is 3, and
the admissibility conditions (\ref{eqn:admissibility-conditions}) reduce to $n \ge 5$, $n \equiv 3 \text{~or~} 5 \adfmod{6}$.
We have $|\mathcal{B}_\mathrm{R}| = {n (n - 5)}/{6}$.
These designs are easy to find if they exist and $n$ is not too large.
When $n = 5$ the graph is $C_5$ and $\mathcal{B}_\mathrm{R} = \{\}$.

\begin{theorem}\label{thm:2-regular-nonexistence}
No design exists for
\begin{equation}\label{eqn:2-regular-nonexistence}
\begin{array}{ll}
     n = 9:  & C_9, \\
     n = 11: & C_{11},~~ C_5 + C_6, \\
     n = 15: & C_6 + C_9,~~ C_7 + C_8, \\
     n = 17: & C_7 + C_{10},~~ C_8 + C_9,~~ C_5 + C_5 + C_7.
\end{array}
\end{equation}
\end{theorem}
\begin{proof}
In Figure~\ref{fig:2-regular-nonexistence-C9-C11-C5+C6} we show the pair occurrence arrays
associated with designs for $C_9$, $C_{11}$ and $C_5 + C_6$.
We assume the vertex set is $\{0, 1, \dots\}$.
We use `{\tt X}' to indicate the pairs of $\mathcal{B}_\mathrm{N}$ and `{\tt -}' for the pairs of $\mathcal{B}_\mathrm{R}$.
\begin{figure}[h]
\begin{minipage}{\textwidth}
{\small\begin{verbatim}
  1 2 3 4 5 6 7 8     1 2 3 4 5 6 7 8 9 A     1 2 3 4 5 6 7 8 9 A
0 X X - - - - X X   0 X X - - - - - - X X   0 X X X X - - - - - -
1   X X - - - - X   1   X X - - - - - - X   1   X X X - - - - - -
2     X X - - - -   2     X X - - - - - -   2     X X - - - - - -
3       X X - - -   3       X X - - - - -   3       X - - - - - -
4         X X - -   4         X X - - - -   4         - - - - - -
5           X X -   5           X X - - -   5           X X - X X
6             X X   6             X X - -   6             X X - X
7               X   7               X X -   7               X X -
                    8                 X X   8                 X X
                    9                   X   9                   X
\end{verbatim}}
\end{minipage}
\caption{Pair occurrence arrays for $C_9$, $C_{11}$ and $C_5 + C_6$}
\label{fig:2-regular-nonexistence-C9-C11-C5+C6}
\end{figure}

From Figure~\ref{fig:2-regular-nonexistence-C9-C11-C5+C6},
non-existence of a design for $C_9$ follows from the observation that pair
$\{0,5\}$ clearly cannot be extended to a block $\{0, 5, x\}$.

For $C_{11}$, pair \{0,5\} must extend to block \{0,5,8\},
then \{0,4\} to \{0,4,7\}, \{0,3\} to \{0,3,6\} and \{1,6\} to \{1,6,9\},
but now pair $\{1,7\}$ cannot be extended.

For $C_5 + C_6$, blocks \{0,5,8\}, \{0,6,9\}, \{0,7,A\} are forced and then pair $\{1,5\}$ cannot be extended.

For the other five graphs in (\ref{eqn:2-regular-nonexistence}), it is not difficult to program exhaustive searches.
In each case it is possible to run the search to completion in a short time.
\end{proof}

We conjecture that designs exist for all admissible 2-regular graphs except those indicated in
Theorem~\ref{thm:2-regular-nonexistence}.
On the other hand, we are able to prove that for each admissible $n \notin \{9, 11\}$,
there exists a design for some 2-regular graph with $n$ vertices;
see Theorem~\ref{thm:designs-2-regular-n-3-5-mod-6}.
Moreover, when $n \equiv 5 \adfmod{6}$ we can prove existence specifically for cycle graphs.

\begin{theorem}\label{thm:designs-2-regular-n-5-mod-6}
A design exists for the cycle graph $C_n$ for every $n \equiv 5 \adfmod{6}$, $n \ge 5$, except for $n = 11$.
\end{theorem}
\begin{proof}
By Theorem~\ref{thm:2-regular-nonexistence}, there is no design for $C_{11}$.

Designs on point set $\{0, 1, \dots, n-1\}$ for $C_n$, $n \in \{5, 17, 23, 29\}$ are generated from
base blocks by $x \mapsto x + 1 \mathrm{~mod~} n$,
as tabulated below.
(Here and in the rest of the paper, the expression $a \mathrm{~mod~} b$ denotes
the integer in \{0, 1, \dots, $b - 1$\} that is congruent to $a$ modulo $b$.)

{\centering\begin{tabular}{ll}
{\boldmath $C_{ 5}$}: &
\{0,1,4\} \\
{\boldmath $C_{17}$}:  &
\{0,1,16\}, \{0,4,10\}, \{0,3,8\} \\
{\boldmath $C_{23}$}:  &
\{0,1,22\}, \{0,3,11\}, \{0,4,9\}, \{0,6,13\} \\
{\boldmath $C_{29}$}:  &
\{0,1,28\}, \{0,3,10\}, \{0,6,15\}, \{0,5,16\}, \{0,4,12\}
\end{tabular}}

Now suppose $m \ge 5$ and $n = 6m+5$.
By Lemma 2.1 of \cite{Forbes2020P} (see also \cite{BermondBrouwerGermaAbe1978} and \cite{Simpson1983}),
there exists a set $T$ of $m$ triples of the form $(0, x, z)$ such that
\begin{align*}
&\bigcup\{\{x, z-x, z\}: (0, x, z) \in T\} \\
  & ~~~~~ = \left\{\begin{array}{ll}\{3, 4, \dots, 3m + 2\} & \textrm{~if~} m \equiv 0,1 \adfmod{4},\\
                                       \{3, 4, \dots, 3m + 1, 3m + 3\} & \textrm{~if~} m \equiv 2,3 \adfmod{4}.
\end{array}\right.
\end{align*}
The design on point set $\{0, 1, \dots, n - 1\}$ for the graph $C_{n}$ is generated by the set of triples
$$ \{\{0, 1, n-1\}\} \cup T $$
under the action of the mapping $x \mapsto x + 1 \mathrm{~mod~} n$.

To see this, first observe that in the triple $\{0, 1, n-1\}$ the difference 2 occurs once,
but the difference 1 occurs twice, corresponding to the edges of the cycle graph,
which are generated from $\{0,1\}$ by $x \mapsto x + 1 \mathrm{~mod~} n$.
Thus $\mathcal{B}_\textrm{N}$ is generated by $\{0,1,n-1\}$.

The remaining differences \{3, 4, \dots, $3m + 2$\} occur precisely once each in $T$, and
therefore $T$ generates $\mathcal{B}_\textrm{R}$.
Note that differences $3m + 3$ and $3m + 2$ are the same.
\end{proof}

\begin{lemma}\label{lem:designs-2-regular-n-3-mod-6-direct}
Designs exist for cycle graphs $C_n$, $n \in \{15, 21, 27, 33, 39\}$.
\end{lemma}
\begin{proof}
Designs on point set $\{0, 1, \dots, n-1\}$ for the stated graphs are generated from
base blocks by $x \mapsto x + s \mathrm{~mod~} n$, as follows:

{\raggedright
{\boldmath $C_{15}$}: $s=5$, \\
 $\{0, 1, 14\}$, $\{0, 1, 2\}$, $\{1, 2, 3\}$, $\{2, 3, 4\}$, $\{3, 4, 5\}$,
 $\{0, 4, 8\}$, $\{1, 4, 9\}$, $\{2, 7, 14\}$, $\{0, 3, 9\}$, $\{0, 7, 11\}$,
 $\{1, 6, 13\}$, $\{2, 8, 13\}$, $\{0, 6, 12\}$, $\{0, 5, 10\}$;

{\boldmath $C_{21}$}: $s=3$, \\
 $\{0, 1, 20\}$, $\{0, 1, 2\}$, $\{1, 2, 3\}$, $\{0, 10, 13\}$, $\{1, 5, 16\}$,
 $\{0, 4, 16\}$, $\{1, 14, 17\}$, $\{0, 5, 11\}$, $\{0, 3, 9\}$, $\{0, 8, 17\}$,
 $\{0, 7, 14\}$;

{\boldmath $C_{27}$}: $s=3$, \\
 $\{0, 1, 26\}$, $\{0, 1, 2\}$, $\{1, 2, 3\}$, $\{0, 4, 15\}$, $\{0, 5, 21\}$,
 $\{0, 3, 10\}$, $\{0, 14, 22\}$, $\{0, 17, 20\}$, $\{0, 13, 19\}$, $\{0, 8, 23\}$,
 $\{1, 4, 17\}$, $\{1, 5, 11\}$, $\{1, 8, 13\}$, $\{0, 9, 18\}$, $\{1, 10, 19\}$,
 $\{2, 11, 20\}$;

{\boldmath $C_{33}$}: $s=3$, \\
 $\{0, 1, 32\}$, $\{0, 1, 2\}$, $\{1, 2, 3\}$, $\{0, 4, 10\}$, $\{0, 7, 25\}$,
 $\{0, 19, 28\}$, $\{1, 5, 22\}$, $\{0, 13, 16\}$, $\{1, 8, 14\}$, $\{1, 11, 26\}$,
 $\{1, 20, 29\}$, $\{0, 5, 17\}$, $\{0, 23, 26\}$, $\{0, 9, 29\}$, $\{0, 6, 14\}$,
 $\{0, 3, 15\}$, $\{0, 11, 22\}$;

{\boldmath $C_{39}$}: $s=3$,\\
 $\{0, 1, 38\}$, $\{0, 1, 2\}$, $\{1, 2, 3\}$, $\{0, 5, 11\}$, $\{0, 8, 29\}$,
 $\{0, 14, 23\}$, $\{0, 20, 32\}$, $\{1, 5, 20\}$, $\{1, 8, 11\}$, $\{0, 4, 35\}$,
 $\{0, 17, 22\}$, $\{1, 4, 29\}$, $\{1, 7, 23\}$, $\{0, 7, 19\}$, $\{0, 25, 34\}$,
 $\{0, 16, 31\}$, $\{0, 10, 28\}$, $\{0, 3, 12\}$, $\{0, 6, 21\}$, $\{0, 13, 26\}$.

} 

For $C_{15}$, the last base block represents a short orbit.
For $C_{27}$, the last three base blocks represent short orbits.
\end{proof}

\begin{theorem}\label{thm:designs-2-regular-n-3-5-mod-6}
For $n \equiv 3 \text{~or~} 5 \adfmod{6}$, $n \ge 5$,
there exists a $2$-regular $n$-vertex graph $G$ with girth at least $5$ for which there is a design of order $n$,
except for $n \in \{9, 11\}$.
\end{theorem}
\begin{proof}
By Theorem~\ref{thm:2-regular-nonexistence}, there is no design of order 9 or 11.
Therefore, by Theorem~\ref{thm:designs-2-regular-n-5-mod-6}, it suffices to deal with $n \equiv 3 \adfmod{6}$, $n \ge 15$.

We use Wilson's fundamental construction, \cite{WilsonRM1972}, \cite{GreigMullen2007}.
Generally, to construct a design of order $g u + m$, we take
a 3-GDD of type $g^u m^1$ and overlay the groups of sizes $g$ and $m$ with designs for $C_g$ and $C_m$.

The details are set out in Table~\ref{tab:constructions-2-regular}.
For the existence of designs for the ingredients specified in column 2 of Table~\ref{tab:constructions-2-regular}, see
Theorem~\ref{thm:designs-2-regular-n-5-mod-6} and Lemma~\ref{lem:designs-2-regular-n-3-mod-6-direct}.
See Lemma~\ref{lem-GDD-existence} for the existence of the specified group divisible designs.

The first five entries in Table~\ref{tab:constructions-2-regular} deal with all sufficiently large orders by residue class modulo 30.
The last three entries mop up those orders missed by the first five.

\begin{table}[h]
\begin{center}
\begin{tabular}{l|l|l}
Design orders constructed & Ingredients & 3-GDD type\\
\hline
{\boldmath\bf $30t +  3$, $t \ge 3$} & $C_{5}$, $C_{33}$ & $5^{6(t-1)} 33^1$ \rule{0mm}{4.5mm}\\ 
{\boldmath\bf $30t +  9$, $t \ge 3$} & $C_{5}$, $C_{39}$ & $5^{6(t-1)} 39^1$ \\                  
{\boldmath\bf $30t + 15$, $t \ge 1$} & $C_{5}$, $C_{15}$ & $5^{6t} 15^1$ \\                      
{\boldmath\bf $30t + 21$, $t \ge 1$} & $C_{5}$, $C_{21}$ & $5^{6t} 21^1$ \\                      
{\boldmath\bf $30t + 27$, $t \ge 2$} & $C_{5}$, $C_{27}$ & $5^{6t} 27^1$ \\[1mm]                      
{\boldmath\bf 57}                    & $C_{5}$, $C_{17}$ & $5^8 17^1$ \\
{\boldmath\bf 63}                    & $C_{21}$          & $21^3$ \\
{\boldmath\bf 69}                    & $C_{23}$          & $23^3$
\end{tabular}
\end{center}
\caption{Constructions for 2-regular graphs}
\label{tab:constructions-2-regular}
\end{table}

Hence there exists a design of order $n$ for all $n \equiv 3 \adfmod{6}$, $n \ge 15$.
\end{proof}

Observe that the 2-regular graphs corresponding to designs constructed using 3-GDDs are not connected.

\section{3-regular graphs}\label{sec:3-regular-graphs}

For {$\delta = 3$}, the block size is 4, and the admissibility conditions reduce to
$n \ge 10$, $n \equiv 4 \adfmod{6}$.
The main result of this section is that a design of order $n$ exists for each admissible $n \ge 40$.

\begin{theorem}\label{thm:designs-3-regular-direct}
For $n \in \{10\} \cup \{40, 46, 52, \dots, 202\}$,
there exists a connected $3$-regular $n$-vertex graph with girth at least $5$ for which there is a design of order $n$.
\end{theorem}
\begin{proof}
For $n = 10$, the design consists of the closed neighbourhoods of the Petersen graph.

For the others, we assume the point set is $\{0, 1, \dots, n\}$.
The blocks of the design are generated from lists of base blocks by $x \mapsto x + s \mathrm{~mod~} n$,
where $s \in \{2, 4\}$.
Base blocks that represent short orbits are shown at the ends of the lists.
There are two short orbits if $n \equiv 8 \adfmod{16}$, four if $n \equiv 0 \adfmod{16}$.

Order 40 uses the generalized Petersen graph $\mathrm{GP}(20,4)$.
Orders greater than 40 use $\mathrm{GP}(n/2,3)$.
There is nothing special about them---a random 3-regular graph with girth at least 5 often works.
For completeness, we include order 10.

{\raggedright
{\bf Order 10}: $s = 2$,
 $\{0, 1, 2, 8\}$, $\{1, 0, 5, 7\}$.

{\bf Order 40}: $s = 4$, \\
 $\{0, 1, 2, 38\}$, $\{0, 1, 9, 33\}$, $\{0, 2, 3, 4\}$, $\{2, 3, 11, 35\}$, $\{1, 3, 14, 30\}$, $\{0, 14, 22, 29\}$, $\{0, 6, 18, 23\}$, $\{1, 6, 27, 29\}$, $\{0, 17, 21, 35\}$, $\{0, 11, 15, 24\}$, $\{0, 7, 25, 28\}$, $\{1, 7, 10, 35\}$, $\{0, 5, 26, 32\}$, $\{0, 10, 20, 30\}$, $\{1, 11, 21, 31\}$.

{\bf Order 46}: $s = 2$, \\
 $\{0, 1, 2, 44\}$, $\{0, 1, 7, 41\}$, $\{0, 5, 28, 38\}$, $\{0, 6, 22, 33\}$, $\{0, 9, 12, 26\}$, $\{0, 21, 25, 35\}$, $\{0, 15, 31, 39\}$, $\{0, 17, 19, 37\}$.

{\bf Order 52}: $s = 4$, \\
 $\{0, 1, 2, 50\}$, $\{0, 1, 7, 47\}$, $\{0, 2, 3, 4\}$, $\{2, 3, 9, 49\}$, $\{0, 5, 36, 42\}$, $\{0, 8, 17, 32\}$, $\{0, 10, 15, 40\}$, $\{0, 11, 18, 34\}$, $\{0, 30, 39, 43\}$, $\{0, 13, 26, 38\}$, $\{0, 14, 31, 46\}$, $\{1, 15, 18, 42\}$, $\{1, 11, 22, 30\}$, $\{1, 3, 25, 34\}$, $\{0, 19, 35, 49\}$, $\{0, 25, 29, 45\}$, $\{0, 23, 33, 41\}$, $\{1, 19, 27, 51\}$.

{\bf Order 58}: $s = 2$, \\
 $\{0, 1, 2, 56\}$, $\{0, 1, 7, 53\}$, $\{0, 5, 42, 52\}$, $\{0, 8, 17, 34\}$, $\{0, 12, 25, 40\}$, $\{0, 19, 35, 38\}$, $\{0, 15, 22, 49\}$, $\{0, 29, 33, 47\}$, $\{0, 14, 37, 45\}$, $\{1, 3, 23, 33\}$.

{\bf Order 64}: $s = 4$, \\
 $\{0, 1, 2, 62\}$, $\{0, 1, 7, 59\}$, $\{0, 2, 3, 4\}$, $\{2, 3, 9, 61\}$, $\{0, 5, 14, 38\}$, $\{0, 6, 40, 50\}$, $\{0, 8, 26, 54\}$, $\{0, 9, 11, 44\}$, $\{0, 12, 25, 33\}$, $\{0, 15, 17, 36\}$, $\{0, 19, 34, 49\}$, $\{0, 39, 42, 53\}$, $\{0, 22, 47, 57\}$, $\{0, 27, 35, 58\}$, $\{0, 23, 41, 61\}$, $\{0, 37, 51, 55\}$, $\{1, 14, 23, 43\}$, $\{1, 26, 31, 38\}$, $\{2, 10, 23, 47\}$, $\{1, 11, 22, 39\}$, $\{1, 5, 29, 46\}$, $\{0, 16, 32, 48\}$, $\{1, 17, 33, 49\}$, $\{2, 18, 34, 50\}$, $\{3, 19, 35, 51\}$.

} 

The others are in Appendix~\ref{app:3-regular-direct}.
\end{proof}

The limit $n = 202$ in Theorem~\ref{thm:designs-3-regular-direct} has no significance.
With sufficient patience and processing power one can take the computations considerably further.

\begin{theorem}\label{lem:designs-3-regular}
For $n \equiv 4 \adfmod{6}$, $n \ge 10$,
there exists a $3$-regular $n$-vertex graph $G$ with girth at least $5$ for which there is a design of order $n$,
except for $n \in \{16, 22\}$ and except possibly for $n \in \{28, 34\}$.
\end{theorem}
\begin{proof}
Using data provided by Meringer \cite{MeringerReggraphs1999} (see also \cite{Meringer1999}),
we tested completely
all forty-nine 3-regular 16-vertex graphs with girth at least 5 and
all 90940      3-regular 22-vertex graphs with girth at least 5 (including the two non-connected ones).
No designs were found.
We therefore claim that there is  no design for any of these graphs.
In Appendix~\ref{app:algorithms} we outline the methods employed.

For orders not covered by Theorem~\ref{thm:designs-3-regular-direct}, we use Wilson's fundamental construction, \cite{WilsonRM1972}, \cite{GreigMullen2007},
to construct designs---as described in the proof of Theorem~\ref{thm:designs-2-regular-n-3-5-mod-6}.
The ingredients for the constructions, namely designs of orders 10, 46, 52, 58 and 64,
exist by Theorem~\ref{thm:designs-3-regular-direct}.
The relevant 4-GDDs exist by Lemma~\ref{lem-GDD-existence}.
Therefore designs of the following orders exist:\\
{\boldmath\bf $30t +  4$, $t \ge 7$} using 4-GDD type $10^{3(t - 2)} 64^1$ (misses 94, 124, 154, 184); \\
{\boldmath\bf $30t + 10$, $t \ge 1$} using 4-GDD type $10^{3t} 10^1$; \\
{\boldmath\bf $30t + 16$, $t \ge 5$} using 4-GDD type $10^{3(t - 1)} 46^1$ (misses 76, 106, 136); \\
{\boldmath\bf $30t + 22$, $t \ge 5$} using 4-GDD type $10^{3(t - 1)} 52^1$ (misses 82, 112, 142); \\
{\boldmath\bf $30t + 28$, $t \ge 6$} using 4-GDD type $10^{3(t - 1)} 58^1$ (misses 88, 118, 148, 178).

The orders missed are provided by Theorem~\ref{thm:designs-3-regular-direct}.
\end{proof}

The 3-regular graphs corresponding to designs constructed using 4-GDDs are not connected.
We have been unable to create a design for a 3-regular graph with 28 or 34 vertices and girth least 5.

\section{(4 or more)-regular graphs}\label{sec:4-or-more-regular-graphs}

We enjoyed limited success at finding new designs when the vertex degree is greater than three.
We know of only four connected $\delta$-regular graphs with $\delta \ge 4$ and girth at least 5
such that the closed neighbourhoods of the vertices can be extended to a regular-graph design with $\lambda = 1$.
They are
the Sylvester graph ($\delta = 5$, used to illustrate the basic construction in Section~\ref{sec:The-basic-construction}),
the Hoffman--Singleton graph ($\delta = 7$, $\mathcal{B}_{\mathrm{R}} = \{\}$),
and the two designs described below.

The design for a 4-regular graph on $n$ vertices $\{0, 1, \dots, n - 1\}$
is generated from base blocks by the mapping $x \mapsto x + s \textrm{~mod~} {n}$.
The first three generate the neighbourhood blocks from which the graph can be recovered.

{\raggedright
{\bf Order 105}: $s = 3$, \\
 $\{0, 53, 88, 98, 100\}$, $\{1, 6, 18, 53, 68\}$, $\{2, 9, 40, 54, 55\}$, $\{87, 83, 57, 26, 20\}$, $\{39, 79, 62, 3, 42\}$, $\{0, 2, 51, 99, 103\}$, $\{0, 5, 18, 27, 34\}$, $\{0, 13, 24, 43, 97\}$, $\{0, 10, 17, 33, 104\}$, $\{0, 14, 64, 80, 90\}$, $\{0, 11, 32, 55, 86\}$, $\{0, 49, 58, 77, 85\}$, $\{0, 28, 41, 61, 67\}$, $\{1, 19, 43, 65, 92\}$, $\{0, 22, 25, 65, 70\}$, $\{1, 2, 26, 35, 38\}$, $\{0, 21, 42, 63, 84\}$.
The last block represents a short orbit.
The graph is connected and has girth 6.

{\bf Order 117}: $s = 3$, \\
 $\{0, 11, 18, 68, 99\}$, $\{1, 43, 76, 101, 107\}$, $\{2, 13, 19, 51, 108\}$, $\{13, 105, 30, 89, 50\}$, $\{7, 86, 14, 42, 51\}$, $\{56, 74, 108, 110, 82\}$, $\{65, 12, 60, 57, 38\}$, $\{0, 61, 104, 113, 116\}$, $\{0, 7, 47, 76, 77\}$, $\{1, 17, 40, 67, 113\}$, $\{0, 14, 46, 56, 107\}$, $\{0, 23, 38, 58, 71\}$, $\{0, 41, 55, 66, 109\}$, $\{0, 40, 52, 70, 74\}$, $\{0, 13, 15, 32, 34\}$, $\{0, 6, 16, 84, 96\}$, $\{0, 1, 4, 24, 54\}$, $\{0, 31, 88, 103, 112\}$.
The graph is connected and has girth 6.

} 

From these one can construct designs of orders
$420t + 105$ for $t \ge 1$,
$2340t + 117$ for $t \ge 1$ and
$2340t + 585$ for $t \ge 0$
by filling in the groups of the 5-GDDs of Lemma~\ref{lem-GDD-existence} item (vi).
The corresponding graphs are not connected.

We claim there is no design for any of the 4131991 4-regular graphs of 25 vertices and girth 5.

\section*{Acknowledgement}

We would like to thank Dr Markus Meringer for making available
the program {\sc Genreg} and the edge sets of regular graphs with girth at least 5, \cite{Meringer1999,MeringerReggraphs1999}.

\section*{ORCID}

\noindent A. D. Forbes     \url{https://orcid.org/0000-0003-3805-7056} \\
C. G. Rutherford           \url{https://orcid.org/0000-0003-1924-207X}


\appendix

\section{Designs for Theorem~\ref{thm:designs-3-regular-direct}}\label{app:3-regular-direct}

Blocks for the stated design order $n$ are generated from the list of base blocks by $x \mapsto x + s \mathrm{~mod~} n$.
There are two short orbits if $n \equiv 8 \adfmod{16}$, four if $n \equiv 0 \adfmod{16}$.

{\raggedright

{\bf Order 70}: $s = 2$, \\
 $\{0, 1, 2, 68\}$, $\{0, 1, 7, 65\}$, $\{0, 5, 54, 64\}$, $\{0, 8, 17, 46\}$, $\{0, 12, 25, 42\}$, $\{0, 15, 26, 49\}$, $\{0, 14, 33, 48\}$, $\{0, 20, 57, 67\}$, $\{0, 18, 45, 61\}$, $\{0, 29, 31, 51\}$, $\{1, 9, 27, 41\}$, $\{0, 35, 39, 63\}$.

{\bf Order 76}: $s = 4$, \\
 $\{0, 1, 2, 74\}$, $\{0, 1, 7, 71\}$, $\{0, 2, 3, 4\}$, $\{2, 3, 9, 73\}$, $\{0, 5, 14, 30\}$, $\{0, 6, 48, 58\}$, $\{0, 8, 26, 62\}$, $\{0, 9, 22, 66\}$, $\{0, 11, 13, 46\}$, $\{0, 15, 42, 70\}$, $\{0, 12, 29, 50\}$, $\{0, 19, 32, 56\}$, $\{0, 21, 23, 40\}$, $\{0, 16, 41, 51\}$, $\{0, 27, 43, 67\}$, $\{0, 33, 47, 73\}$, $\{0, 37, 55, 69\}$, $\{0, 45, 49, 65\}$, $\{0, 31, 53, 61\}$, $\{1, 18, 31, 49\}$, $\{1, 25, 54, 62\}$, $\{1, 23, 42, 67\}$, $\{1, 27, 35, 66\}$, $\{1, 39, 43, 50\}$, $\{2, 11, 14, 31\}$, $\{2, 7, 22, 55\}$.

{\bf Order 82}: $s = 2$, \\
 $\{0, 1, 2, 80\}$, $\{0, 1, 7, 77\}$, $\{0, 5, 66, 76\}$, $\{0, 8, 17, 58\}$, $\{0, 12, 25, 54\}$, $\{0, 14, 34, 60\}$, $\{0, 15, 38, 57\}$, $\{0, 23, 30, 61\}$, $\{0, 18, 51, 65\}$, $\{0, 29, 39, 55\}$, $\{0, 27, 45, 73\}$, $\{0, 49, 71, 79\}$, $\{0, 43, 63, 67\}$, $\{0, 35, 37, 69\}$.

{\bf Order 88}: $s = 4$, \\
 $\{0, 1, 2, 86\}$, $\{0, 1, 7, 83\}$, $\{0, 2, 3, 4\}$, $\{2, 3, 9, 85\}$, $\{55, 23, 80, 68\}$, $\{42, 32, 55, 3\}$, $\{25, 43, 74, 86\}$, $\{42, 23, 14, 60\}$, $\{55, 8, 65, 82\}$, $\{0, 6, 27, 35\}$, $\{0, 5, 15, 39\}$, $\{0, 9, 67, 71\}$, $\{1, 3, 5, 75\}$, $\{1, 15, 22, 55\}$, $\{0, 14, 19, 79\}$, $\{1, 10, 27, 47\}$, $\{1, 26, 51, 66\}$, $\{1, 31, 42, 78\}$, $\{0, 11, 46, 61\}$, $\{0, 13, 55, 58\}$, $\{0, 18, 26, 82\}$, $\{0, 30, 50, 81\}$, $\{0, 25, 38, 54\}$, $\{0, 16, 45, 78\}$, $\{0, 34, 53, 69\}$, $\{0, 17, 37, 77\}$, $\{0, 21, 24, 73\}$, $\{0, 8, 36, 56\}$, $\{0, 33, 41, 65\}$, $\{0, 22, 44, 66\}$, $\{1, 23, 45, 67\}$.

{\bf Order 94}: $s = 2$, \\
 $\{0, 1, 2, 92\}$, $\{0, 1, 7, 89\}$, $\{0, 5, 78, 88\}$, $\{0, 8, 17, 70\}$, $\{0, 12, 25, 56\}$, $\{0, 14, 29, 72\}$, $\{0, 18, 37, 66\}$, $\{0, 20, 43, 60\}$, $\{0, 27, 30, 69\}$, $\{0, 31, 52, 87\}$, $\{0, 26, 71, 75\}$, $\{0, 53, 55, 85\}$, $\{0, 33, 61, 79\}$, $\{0, 59, 67, 81\}$, $\{0, 47, 57, 83\}$, $\{1, 17, 41, 61\}$.

{\bf Order 100}: $s = 4$, \\
 $\{0, 1, 2, 98\}$, $\{0, 1, 7, 95\}$, $\{0, 2, 3, 4\}$, $\{2, 3, 9, 97\}$, $\{61, 19, 84, 52\}$, $\{8, 80, 59, 46\}$, $\{0, 5, 80, 86\}$, $\{0, 8, 18, 64\}$, $\{0, 11, 24, 37\}$, $\{0, 12, 26, 52\}$, $\{0, 15, 17, 84\}$, $\{0, 19, 22, 65\}$, $\{0, 21, 23, 85\}$, $\{0, 29, 39, 89\}$, $\{0, 30, 42, 53\}$, $\{0, 27, 34, 81\}$, $\{0, 41, 50, 55\}$, $\{0, 45, 49, 69\}$, $\{0, 58, 75, 93\}$, $\{0, 59, 63, 97\}$, $\{0, 57, 73, 94\}$, $\{0, 47, 62, 78\}$, $\{0, 43, 70, 90\}$, $\{0, 61, 83, 91\}$, $\{0, 46, 71, 82\}$, $\{1, 18, 33, 62\}$, $\{1, 14, 46, 73\}$, $\{1, 34, 43, 79\}$, $\{1, 9, 27, 57\}$, $\{1, 26, 50, 87\}$, $\{1, 35, 70, 91\}$, $\{2, 10, 43, 62\}$, $\{3, 19, 43, 71\}$, $\{2, 30, 59, 79\}$.

{\bf Order 106}: $s = 2$, \\
 $\{0, 1, 2, 104\}$, $\{0, 1, 7, 101\}$, $\{78, 103, 24, 6\}$, $\{0, 5, 60, 100\}$, $\{0, 8, 17, 90\}$, $\{0, 10, 22, 78\}$, $\{0, 13, 70, 85\}$, $\{0, 19, 44, 86\}$, $\{0, 21, 23, 58\}$, $\{0, 27, 74, 103\}$, $\{0, 14, 45, 87\}$, $\{0, 26, 61, 83\}$, $\{0, 30, 77, 93\}$, $\{0, 43, 53, 89\}$, $\{0, 41, 65, 91\}$, $\{0, 67, 95, 99\}$, $\{1, 9, 49, 63\}$, $\{0, 37, 55, 75\}$.

{\bf Order 112}: $s = 4$, \\
 $\{0, 1, 2, 110\}$, $\{0, 1, 7, 107\}$, $\{0, 2, 3, 4\}$, $\{2, 3, 9, 109\}$, $\{31, 88, 81, 23\}$, $\{66, 98, 51, 88\}$, $\{110, 43, 76, 16\}$, $\{13, 100, 21, 8\}$, $\{23, 0, 85, 106\}$, $\{3, 63, 33, 94\}$, $\{59, 42, 25, 28\}$, $\{83, 109, 45, 16\}$, $\{43, 65, 78, 91\}$, $\{91, 100, 42, 6\}$, $\{110, 81, 21, 32\}$, $\{25, 56, 64, 39\}$, $\{47, 32, 8, 49\}$, $\{0, 12, 38, 71\}$, $\{0, 6, 32, 96\}$, $\{0, 9, 19, 68\}$, $\{0, 21, 61, 76\}$, $\{0, 37, 62, 102\}$, $\{0, 30, 42, 72\}$, $\{0, 58, 77, 99\}$, $\{0, 50, 66, 74\}$, $\{0, 43, 46, 83\}$, $\{0, 45, 89, 91\}$, $\{0, 35, 51, 69\}$, $\{0, 11, 65, 98\}$, $\{2, 71, 75, 95\}$, $\{2, 11, 55, 91\}$, $\{1, 46, 98, 103\}$, $\{1, 19, 33, 70\}$, $\{1, 25, 67, 74\}$, $\{1, 5, 75, 86\}$, $\{1, 17, 58, 93\}$, $\{1, 10, 54, 102\}$, $\{0, 28, 56, 84\}$, $\{1, 29, 57, 85\}$, $\{2, 30, 58, 86\}$, $\{3, 31, 59, 87\}$.

{\bf Order 118}: $s = 2$, \\
 $\{0, 1, 2, 116\}$, $\{0, 1, 7, 113\}$, $\{53, 109, 72, 2\}$, $\{68, 25, 33, 4\}$, $\{37, 59, 69, 111\}$, $\{0, 5, 63, 112\}$, $\{0, 8, 25, 53\}$, $\{0, 9, 91, 108\}$, $\{0, 12, 27, 67\}$, $\{0, 13, 61, 85\}$, $\{0, 14, 47, 73\}$, $\{0, 18, 95, 115\}$, $\{1, 17, 51, 81\}$, $\{0, 39, 43, 57\}$, $\{0, 16, 103, 105\}$, $\{0, 28, 68, 109\}$, $\{0, 31, 38, 84\}$, $\{0, 22, 58, 93\}$, $\{0, 23, 62, 92\}$, $\{0, 20, 44, 86\}$.

{\bf Order 124}: $s = 4$, \\
 $\{0, 1, 2, 122\}$, $\{0, 1, 7, 119\}$, $\{0, 2, 3, 4\}$, $\{2, 3, 9, 121\}$, $\{122, 108, 31, 17\}$, $\{81, 121, 6, 47\}$, $\{20, 95, 35, 55\}$, $\{43, 70, 110, 38\}$, $\{63, 74, 86, 15\}$, $\{93, 91, 8, 67\}$, $\{0, 5, 23, 39\}$, $\{0, 6, 19, 87\}$, $\{0, 8, 71, 103\}$, $\{0, 9, 51, 55\}$, $\{0, 11, 18, 107\}$, $\{0, 29, 31, 111\}$, $\{0, 10, 27, 79\}$, $\{1, 5, 27, 63\}$, $\{0, 13, 91, 99\}$, $\{0, 21, 54, 115\}$, $\{0, 22, 37, 67\}$, $\{0, 34, 43, 94\}$, $\{1, 9, 75, 90\}$, $\{1, 11, 14, 114\}$, $\{1, 18, 86, 111\}$, $\{1, 70, 78, 107\}$, $\{1, 55, 74, 94\}$, $\{1, 33, 71, 81\}$, $\{1, 26, 54, 103\}$, $\{0, 12, 38, 86\}$, $\{0, 16, 62, 106\}$, $\{0, 30, 89, 118\}$, $\{0, 61, 82, 98\}$, $\{0, 24, 66, 76\}$, $\{0, 50, 64, 117\}$, $\{0, 20, 78, 101\}$, $\{1, 21, 62, 89\}$, $\{0, 25, 70, 92\}$, $\{0, 44, 93, 109\}$, $\{0, 69, 97, 121\}$, $\{0, 17, 36, 77\}$, $\{0, 28, 73, 84\}$.

{\bf Order 130}: $s = 2$, \\
 $\{0, 1, 2, 128\}$, $\{0, 1, 7, 125\}$, $\{22, 44, 103, 90\}$, $\{0, 93, 70, 29\}$, $\{49, 102, 35, 64\}$, $\{120, 103, 61, 99\}$, $\{0, 5, 45, 124\}$, $\{0, 8, 18, 32\}$, $\{0, 9, 96, 111\}$, $\{0, 12, 28, 54\}$, $\{0, 17, 40, 61\}$, $\{0, 19, 58, 110\}$, $\{0, 25, 56, 83\}$, $\{0, 30, 66, 97\}$, $\{0, 33, 44, 85\}$, $\{0, 37, 50, 105\}$, $\{0, 47, 79, 82\}$, $\{0, 35, 53, 69\}$, $\{0, 57, 65, 87\}$, $\{0, 49, 103, 123\}$, $\{0, 73, 75, 121\}$, $\{1, 11, 37, 61\}$.

{\bf Order 136}: $s = 4$, \\
 $\{0, 1, 2, 134\}$, $\{0, 1, 7, 131\}$, $\{0, 2, 3, 4\}$, $\{2, 3, 9, 133\}$, $\{79, 28, 83, 94\}$, $\{41, 24, 114, 85\}$, $\{88, 17, 93, 48\}$, $\{64, 126, 13, 0\}$, $\{92, 128, 107, 38\}$, $\{85, 106, 63, 65\}$, $\{31, 83, 122, 109\}$, $\{37, 90, 55, 74\}$, $\{93, 55, 64, 101\}$, $\{48, 70, 97, 111\}$, $\{67, 37, 7, 54\}$, $\{0, 6, 21, 93\}$, $\{0, 8, 33, 81\}$, $\{0, 10, 69, 121\}$, $\{0, 9, 16, 117\}$, $\{0, 11, 77, 109\}$, $\{0, 23, 97, 113\}$, $\{0, 39, 53, 133\}$, $\{1, 5, 41, 86\}$, $\{1, 3, 10, 113\}$, $\{0, 26, 57, 107\}$, $\{0, 18, 48, 89\}$, $\{0, 43, 71, 125\}$, $\{1, 23, 67, 87\}$, $\{1, 58, 70, 75\}$, $\{1, 79, 95, 119\}$, $\{1, 62, 94, 102\}$, $\{1, 98, 118, 127\}$, $\{1, 11, 90, 126\}$, $\{0, 19, 70, 119\}$, $\{0, 27, 59, 86\}$, $\{0, 30, 95, 118\}$, $\{0, 35, 83, 91\}$, $\{2, 30, 63, 94\}$, $\{0, 58, 111, 114\}$, $\{0, 31, 92, 130\}$, $\{0, 14, 74, 98\}$, $\{0, 28, 80, 122\}$, $\{0, 24, 78, 103\}$, $\{0, 47, 60, 110\}$, $\{0, 12, 32, 99\}$, $\{0, 34, 68, 102\}$, $\{1, 35, 69, 103\}$.

{\bf Order 142}: $s = 2$, \\
 $\{0, 1, 2, 140\}$, $\{0, 1, 7, 137\}$, $\{110, 96, 63, 41\}$, $\{8, 58, 2, 70\}$, $\{122, 141, 57, 23\}$, $\{112, 80, 53, 101\}$, $\{12, 60, 124, 23\}$, $\{10, 76, 37, 63\}$, $\{0, 5, 9, 102\}$, $\{0, 8, 24, 108\}$, $\{0, 10, 46, 98\}$, $\{0, 13, 70, 93\}$, $\{0, 17, 26, 55\}$, $\{0, 25, 28, 61\}$, $\{0, 18, 38, 69\}$, $\{0, 15, 59, 82\}$, $\{0, 22, 57, 123\}$, $\{0, 63, 71, 127\}$, $\{0, 97, 99, 113\}$, $\{0, 81, 91, 111\}$, $\{0, 47, 65, 117\}$, $\{0, 39, 89, 121\}$, $\{0, 67, 107, 135\}$, $\{0, 37, 79, 125\}$.

{\bf Order 148}: $s = 4$, \\
 $\{0, 1, 2, 146\}$, $\{0, 1, 7, 143\}$, $\{0, 2, 3, 4\}$, $\{2, 3, 9, 145\}$, $\{87, 2, 53, 129\}$, $\{85, 30, 10, 126\}$, $\{20, 93, 108, 97\}$, $\{21, 4, 82, 124\}$, $\{30, 90, 73, 71\}$, $\{88, 128, 65, 107\}$, $\{131, 83, 100, 1\}$, $\{130, 105, 19, 115\}$, $\{5, 125, 41, 73\}$, $\{59, 115, 93, 32\}$, $\{63, 146, 14, 24\}$, $\{129, 32, 46, 8\}$, $\{40, 126, 58, 3\}$, $\{47, 98, 117, 119\}$, $\{4, 19, 127, 134\}$, $\{35, 70, 44, 22\}$, $\{0, 5, 35, 95\}$, $\{0, 6, 11, 79\}$, $\{0, 8, 51, 115\}$, $\{0, 9, 23, 67\}$, $\{0, 10, 55, 87\}$, $\{0, 21, 47, 71\}$, $\{0, 29, 99, 103\}$, $\{1, 10, 99, 135\}$, $\{0, 22, 75, 91\}$, $\{1, 17, 103, 111\}$, $\{1, 19, 47, 129\}$, $\{0, 13, 16, 135\}$, $\{2, 14, 31, 122\}$, $\{0, 63, 94, 102\}$, $\{2, 11, 38, 139\}$, $\{1, 55, 58, 122\}$, $\{1, 34, 78, 139\}$, $\{0, 59, 81, 89\}$, $\{0, 69, 114, 129\}$, $\{0, 41, 48, 90\}$, $\{1, 38, 49, 105\}$, $\{1, 14, 70, 109\}$, $\{0, 58, 93, 117\}$, $\{0, 30, 53, 96\}$, $\{0, 34, 65, 112\}$, $\{0, 50, 74, 80\}$, $\{0, 20, 57, 92\}$, $\{0, 12, 66, 116\}$, $\{0, 25, 64, 110\}$, $\{0, 33, 62, 134\}$.

{\bf Order 154}: $s = 2$, \\
 $\{0, 1, 2, 152\}$, $\{0, 1, 7, 149\}$, $\{53, 114, 90, 128\}$, $\{67, 32, 86, 15\}$, $\{111, 113, 19, 68\}$, $\{72, 118, 137, 87\}$, $\{64, 148, 139, 70\}$, $\{13, 2, 36, 83\}$, $\{103, 69, 61, 6\}$, $\{52, 148, 65, 19\}$, $\{0, 5, 18, 27\}$, $\{0, 8, 30, 56\}$, $\{0, 10, 62, 74\}$, $\{0, 17, 112, 133\}$, $\{0, 23, 50, 82\}$, $\{0, 20, 60, 88\}$, $\{0, 29, 110, 147\}$, $\{0, 16, 57, 115\}$, $\{0, 36, 85, 103\}$, $\{0, 89, 109, 119\}$, $\{0, 51, 91, 139\}$, $\{0, 101, 125, 129\}$, $\{0, 33, 111, 143\}$, $\{0, 31, 87, 113\}$, $\{0, 39, 53, 107\}$, $\{0, 61, 77, 151\}$.

{\bf Order 160}: $s = 4$, \\
 $\{0, 1, 2, 158\}$, $\{0, 1, 7, 155\}$, $\{0, 2, 3, 4\}$, $\{2, 3, 9, 157\}$, $\{95, 103, 128, 71\}$, $\{129, 35, 71, 121\}$, $\{21, 89, 17, 78\}$, $\{23, 140, 14, 153\}$, $\{50, 107, 39, 110\}$, $\{24, 140, 106, 152\}$, $\{78, 11, 92, 157\}$, $\{14, 2, 84, 49\}$, $\{159, 117, 30, 53\}$, $\{92, 70, 100, 24\}$, $\{53, 103, 120, 2\}$, $\{131, 20, 82, 121\}$, $\{142, 78, 104, 152\}$, $\{32, 117, 96, 141\}$, $\{8, 34, 75, 140\}$, $\{102, 85, 33, 65\}$, $\{136, 42, 35, 70\}$, $\{53, 0, 139, 22\}$, $\{66, 19, 128, 85\}$, $\{0, 5, 18, 102\}$, $\{0, 6, 30, 118\}$, $\{0, 14, 19, 154\}$, $\{1, 3, 30, 146\}$, $\{0, 11, 50, 106\}$, $\{1, 17, 42, 50\}$, $\{0, 58, 74, 87\}$, $\{0, 15, 70, 122\}$, $\{1, 10, 37, 102\}$, $\{0, 25, 110, 142\}$, $\{1, 39, 54, 90\}$, $\{0, 10, 55, 71\}$, $\{1, 27, 46, 115\}$, $\{2, 23, 79, 99\}$, $\{1, 19, 29, 106\}$, $\{1, 71, 94, 131\}$, $\{2, 35, 83, 87\}$, $\{0, 9, 63, 91\}$, $\{0, 23, 41, 52\}$, $\{0, 49, 83, 147\}$, $\{0, 27, 29, 36\}$, $\{0, 31, 75, 145\}$, $\{0, 57, 115, 141\}$, $\{0, 35, 77, 133\}$, $\{0, 24, 123, 137\}$, $\{0, 97, 119, 157\}$, $\{0, 39, 73, 121\}$, $\{0, 16, 72, 105\}$, $\{0, 17, 61, 140\}$, $\{0, 47, 69, 100\}$, $\{0, 40, 80, 120\}$, $\{1, 41, 81, 121\}$, $\{2, 42, 82, 122\}$, $\{3, 43, 83, 123\}$.

{\bf Order 166}: $s = 2$, \\
 $\{0, 1, 2, 164\}$, $\{0, 1, 7, 161\}$, $\{30, 144, 13, 114\}$, $\{54, 66, 147, 120\}$, $\{97, 39, 10, 31\}$, $\{69, 17, 59, 50\}$, $\{23, 142, 92, 68\}$, $\{5, 59, 153, 36\}$, $\{127, 156, 161, 51\}$, $\{164, 117, 40, 100\}$, $\{151, 22, 110, 67\}$, $\{111, 150, 143, 164\}$, $\{0, 6, 16, 138\}$, $\{0, 11, 62, 70\}$, $\{0, 13, 15, 80\}$, $\{0, 18, 43, 146\}$, $\{0, 31, 56, 109\}$, $\{0, 46, 95, 118\}$, $\{0, 26, 131, 134\}$, $\{0, 40, 125, 151\}$, $\{0, 57, 76, 155\}$, $\{0, 59, 68, 139\}$, $\{0, 22, 55, 91\}$, $\{0, 36, 75, 103\}$, $\{0, 51, 73, 89\}$, $\{0, 37, 83, 153\}$, $\{1, 5, 25, 65\}$, $\{1, 15, 45, 119\}$.

{\bf Order 172}: $s = 4$, \\
 $\{0, 1, 2, 170\}$, $\{0, 1, 7, 167\}$, $\{0, 2, 3, 4\}$, $\{2, 3, 9, 169\}$, $\{110, 163, 11, 70\}$, $\{84, 22, 119, 81\}$, $\{65, 101, 152, 134\}$, $\{61, 32, 136, 76\}$, $\{48, 117, 94, 106\}$, $\{142, 1, 108, 134\}$, $\{1, 54, 44, 17\}$, $\{57, 157, 80, 171\}$, $\{141, 37, 114, 162\}$, $\{129, 48, 65, 91\}$, $\{129, 171, 109, 13\}$, $\{144, 104, 91, 46\}$, $\{55, 131, 123, 59\}$, $\{48, 102, 31, 53\}$, $\{4, 146, 115, 139\}$, $\{83, 139, 33, 81\}$, $\{110, 130, 119, 88\}$, $\{102, 36, 163, 14\}$, $\{96, 85, 12, 31\}$, $\{29, 58, 94, 123\}$, $\{171, 117, 108, 87\}$, $\{143, 107, 94, 5\}$, $\{149, 37, 123, 158\}$, $\{59, 152, 75, 41\}$, $\{0, 6, 25, 49\}$, $\{0, 8, 113, 141\}$, $\{0, 11, 53, 57\}$, $\{0, 16, 109, 153\}$, $\{0, 27, 37, 89\}$, $\{0, 13, 21, 101\}$, $\{1, 14, 141, 158\}$, $\{0, 15, 33, 165\}$, $\{0, 23, 62, 125\}$, $\{1, 26, 67, 86\}$, $\{1, 31, 83, 98\}$, $\{1, 62, 94, 99\}$, $\{1, 79, 82, 106\}$, $\{1, 42, 75, 118\}$, $\{1, 58, 74, 138\}$, $\{0, 45, 55, 146\}$, $\{0, 41, 96, 163\}$, $\{1, 47, 91, 171\}$, $\{2, 59, 87, 119\}$, $\{0, 14, 83, 123\}$, $\{0, 20, 59, 100\}$, $\{0, 50, 64, 139\}$, $\{0, 24, 71, 144\}$, $\{0, 12, 98, 115\}$, $\{0, 30, 82, 124\}$, $\{0, 18, 118, 143\}$, $\{0, 32, 122, 166\}$, $\{0, 36, 87, 138\}$, $\{0, 70, 126, 147\}$, $\{0, 38, 106, 116\}$.

{\bf Order 178}: $s = 2$, \\
 $\{0, 1, 2, 176\}$, $\{0, 1, 7, 173\}$, $\{27, 155, 22, 102\}$, $\{173, 73, 139, 154\}$, $\{76, 142, 152, 168\}$, $\{122, 94, 48, 83\}$, $\{85, 149, 26, 177\}$, $\{143, 126, 3, 33\}$, $\{38, 13, 108, 17\}$, $\{144, 20, 34, 164\}$, $\{171, 142, 101, 25\}$, $\{77, 40, 37, 166\}$, $\{0, 6, 30, 38\}$, $\{0, 9, 11, 88\}$, $\{0, 13, 120, 135\}$, $\{0, 18, 60, 96\}$, $\{0, 27, 114, 145\}$, $\{0, 21, 39, 106\}$, $\{0, 12, 56, 79\}$, $\{0, 25, 128, 161\}$, $\{0, 45, 94, 141\}$, $\{0, 81, 107, 165\}$, $\{0, 22, 95, 143\}$, $\{0, 40, 109, 117\}$, $\{0, 51, 65, 116\}$, $\{0, 57, 119, 155\}$, $\{0, 125, 147, 171\}$, $\{0, 41, 115, 131\}$, $\{0, 105, 149, 159\}$, $\{0, 43, 63, 169\}$.

{\bf Order 184}: $s = 4$, \\
 $\{0, 1, 2, 182\}$, $\{0, 1, 7, 179\}$, $\{0, 2, 3, 4\}$, $\{2, 3, 9, 181\}$, $\{5, 99, 86, 130\}$, $\{93, 177, 102, 49\}$, $\{122, 88, 98, 174\}$, $\{180, 155, 18, 92\}$, $\{75, 146, 129, 95\}$, $\{35, 53, 139, 135\}$, $\{182, 166, 99, 51\}$, $\{40, 29, 125, 144\}$, $\{183, 68, 132, 141\}$, $\{72, 64, 10, 154\}$, $\{45, 125, 32, 165\}$, $\{172, 7, 59, 135\}$, $\{87, 15, 64, 57\}$, $\{141, 52, 19, 136\}$, $\{74, 18, 60, 105\}$, $\{38, 129, 131, 35\}$, $\{148, 139, 71, 5\}$, $\{95, 70, 152, 10\}$, $\{132, 46, 181, 143\}$, $\{8, 36, 111, 147\}$, $\{3, 5, 165, 38\}$, $\{179, 153, 77, 145\}$, $\{179, 70, 80, 100\}$, $\{67, 95, 34, 153\}$, $\{98, 30, 126, 175\}$, $\{0, 112, 149, 145\}$, $\{56, 71, 133, 111\}$, $\{0, 6, 131, 163\}$, $\{0, 18, 35, 59\}$, $\{0, 17, 87, 95\}$, $\{0, 26, 47, 167\}$, $\{0, 27, 38, 43\}$, $\{0, 25, 83, 143\}$, $\{0, 31, 50, 171\}$, $\{0, 66, 81, 155\}$, $\{0, 54, 62, 119\}$, $\{0, 91, 106, 117\}$, $\{1, 15, 38, 149\}$, $\{1, 26, 49, 155\}$, $\{0, 78, 113, 123\}$, $\{1, 86, 115, 122\}$, $\{1, 19, 29, 166\}$, $\{0, 39, 118, 150\}$, $\{1, 46, 55, 158\}$, $\{1, 39, 53, 113\}$, $\{1, 21, 42, 90\}$, $\{0, 57, 70, 134\}$, $\{0, 58, 101, 162\}$, $\{0, 74, 129, 158\}$, $\{1, 33, 134, 146\}$, $\{0, 52, 146, 166\}$, $\{0, 21, 126, 140\}$, $\{0, 30, 36, 97\}$, $\{0, 29, 60, 108\}$, $\{0, 53, 56, 68\}$, $\{0, 32, 141, 157\}$, $\{0, 16, 137, 160\}$, $\{0, 46, 92, 138\}$, $\{1, 47, 93, 139\}$.

{\bf Order 190}: $s = 2$, \\
 $\{0, 1, 2, 188\}$, $\{0, 1, 7, 185\}$, $\{84, 113, 151, 162\}$, $\{96, 187, 20, 141\}$, $\{66, 38, 25, 155\}$, $\{20, 88, 163, 148\}$, $\{4, 123, 23, 39\}$, $\{99, 33, 126, 109\}$, $\{163, 133, 48, 105\}$, $\{171, 103, 146, 21\}$, $\{168, 120, 179, 126\}$, $\{44, 14, 148, 94\}$, $\{179, 157, 143, 93\}$, $\{93, 117, 76, 21\}$, $\{0, 5, 175, 182\}$, $\{0, 9, 101, 178\}$, $\{0, 10, 33, 181\}$, $\{0, 14, 51, 125\}$, $\{0, 16, 99, 169\}$, $\{0, 18, 61, 105\}$, $\{0, 55, 73, 118\}$, $\{0, 47, 79, 81\}$, $\{0, 109, 161, 187\}$, $\{1, 9, 57, 111\}$, $\{0, 27, 31, 84\}$, $\{0, 26, 95, 157\}$, $\{0, 20, 123, 158\}$, $\{0, 44, 93, 144\}$, $\{0, 22, 88, 151\}$, $\{0, 36, 107, 132\}$, $\{0, 24, 64, 98\}$, $\{0, 38, 77, 108\}$.

{\bf Order 196}: $s = 4$, \\
 $\{0, 1, 2, 194\}$, $\{0, 1, 7, 191\}$, $\{0, 2, 3, 4\}$, $\{2, 3, 9, 193\}$, $\{35, 43, 135, 71\}$, $\{133, 183, 179, 31\}$, $\{119, 62, 90, 70\}$, $\{26, 20, 98, 192\}$, $\{171, 41, 178, 81\}$, $\{142, 131, 8, 148\}$, $\{87, 5, 111, 42\}$, $\{66, 15, 185, 150\}$, $\{80, 180, 121, 26\}$, $\{154, 81, 125, 170\}$, $\{102, 37, 146, 53\}$, $\{110, 165, 183, 185\}$, $\{90, 158, 38, 136\}$, $\{129, 9, 78, 4\}$, $\{134, 156, 31, 29\}$, $\{17, 47, 89, 125\}$, $\{166, 178, 59, 116\}$, $\{46, 179, 163, 182\}$, $\{16, 142, 103, 48\}$, $\{2, 84, 141, 7\}$, $\{171, 33, 88, 43\}$, $\{54, 64, 22, 85\}$, $\{77, 191, 86, 30\}$, $\{18, 122, 108, 139\}$, $\{152, 104, 37, 139\}$, $\{112, 124, 188, 72\}$, $\{144, 39, 25, 193\}$, $\{88, 67, 58, 96\}$, $\{6, 112, 59, 46\}$, $\{58, 167, 43, 20\}$, $\{126, 101, 15, 4\}$, $\{193, 60, 29, 0\}$, $\{0, 13, 54, 170\}$, $\{0, 9, 70, 178\}$, $\{0, 10, 34, 53\}$, $\{0, 18, 33, 66\}$, $\{0, 15, 46, 82\}$, $\{0, 19, 118, 182\}$, $\{1, 5, 18, 118\}$, $\{0, 43, 63, 86\}$, $\{1, 22, 87, 137\}$, $\{0, 58, 95, 169\}$, $\{0, 115, 162, 185\}$, $\{2, 23, 103, 143\}$, $\{0, 17, 146, 187\}$, $\{0, 28, 67, 138\}$, $\{1, 67, 99, 126\}$, $\{1, 54, 79, 167\}$, $\{0, 26, 37, 107\}$, $\{0, 99, 117, 155\}$, $\{0, 44, 103, 149\}$, $\{0, 47, 101, 159\}$, $\{0, 27, 92, 171\}$, $\{1, 35, 49, 57\}$, $\{0, 75, 113, 135\}$, $\{0, 16, 127, 161\}$, $\{1, 55, 65, 117\}$, $\{0, 51, 84, 173\}$, $\{0, 36, 109, 124\}$, $\{0, 25, 121, 128\}$, $\{0, 20, 65, 177\}$, $\{0, 61, 85, 153\}$.

{\bf Order 202}: $s = 2$, \\
 $\{0, 1, 2, 200\}$, $\{0, 1, 7, 197\}$, $\{10, 42, 101, 189\}$, $\{191, 163, 28, 134\}$, $\{106, 112, 81, 189\}$, $\{161, 175, 68, 21\}$, $\{21, 51, 96, 182\}$, $\{51, 131, 16, 140\}$, $\{190, 16, 67, 27\}$, $\{28, 188, 73, 95\}$, $\{171, 6, 148, 11\}$, $\{106, 98, 142, 60\}$, $\{149, 39, 175, 125\}$, $\{49, 128, 153, 11\}$, $\{125, 81, 166, 116\}$, $\{53, 154, 137, 201\}$, $\{0, 12, 33, 43\}$, $\{0, 10, 27, 73\}$, $\{0, 13, 81, 181\}$, $\{0, 14, 69, 151\}$, $\{0, 15, 89, 121\}$, $\{0, 19, 37, 149\}$, $\{0, 139, 191, 195\}$, $\{0, 53, 129, 144\}$, $\{0, 34, 131, 133\}$, $\{0, 153, 169, 189\}$, $\{0, 61, 119, 138\}$, $\{0, 95, 103, 173\}$, $\{0, 18, 84, 159\}$, $\{0, 49, 52, 108\}$, $\{0, 30, 70, 175\}$, $\{0, 16, 90, 114\}$, $\{0, 20, 68, 130\}$, $\{0, 22, 76, 102\}$.

} 

\section{Algorithms}\label{app:algorithms}

For all forty-nine 3-regular graphs with 16 vertices and girth at least 5, \cite{MeringerReggraphs1999},
we found that Algorithm A is sufficient to prove that no designs exist.
Alternatively, we reached the same conclusion with Algorithm C.

For 3-regular graphs with 22 vertices and girth at least 5, \cite{MeringerReggraphs1999},
we found that Algorithm A establishes non-existence of a design for most but not all of the 90940 graphs. 
So we settled non-existence for all of the graphs in two different ways: Algorithms B and D.
The second iteration of Algorithm D is needed for some graphs. 

Algorithm C suffices to prove non-existence of the designs for
all of the 4131991 4-regular graphs with 25 vertices and girth 5,
as created by {\sc Genreg}, \cite{MeringerReggraphs1999}.

\subsection*{Algorithm A}~ 

Consider a $\delta$-regular graph $G$ with $n$ vertices and girth at least 5.
Assume the graph's vertices are $\{1, 2, \dots, n\}$.

Construct the neighbourhood blocks $\mathcal{B}_{\mathrm{N}}$
and then the set of remainder pairs,
$$\mathcal{P}_{\mathrm{R}} = \{\{i,j\}: 1 \le i < j \le n,~ \{i,j\} \not\subset B \text{ for any } B \in \mathcal{B}_{\mathrm{N}}\}.$$
For each ($\delta + 1$)-tuple
$$q \in Q = \{\{a_1, a_2, \dots, a_{\delta + 1}\}: 1 \le a_1 < a_2 < \dots < a_{\delta + 1} \le n\},$$
compute
$$K(q) = \{\{i,j\}: \{i,j\} \subset q \text{ and } \{i,j\} \in \mathcal{P}_{\mathrm{R}}\}.$$
Discard those $K(q)$ with fewer than $\delta(\delta + 1)/2$ elements to leave
$$\mathcal{K} = \{K(q): q \in Q,~ |K(q)| = \delta(\delta + 1)/2\}.$$
If $|\bigcup \mathcal{K}| < |\mathcal{P}_{\mathrm{R}}|$,
we cannot construct the blocks of $\mathcal{B}_{\mathrm{R}}$ from $\mathcal{K}$, and therefore a design for $G$ does not exist.

\subsection*{Algorithm B}~ 

Consider a $\delta$-regular graph with vertices $\{1, 2, \dots, n\}$ and girth at least 5.
Construct $\mathcal{P}_{\mathrm{R}}$ and $\mathcal{K}$ as in Algorithm A.
If $|\bigcup \mathcal{K}| < |\mathcal{P}_{\mathrm{R}}|$, the design does not exist.
Otherwise continue as follows.

Create the
$|\mathcal{P}_{\mathrm{R}}| \times |\mathcal{K}|$ matrix $M$ with
rows indexed by $\mathcal{P}_{\mathrm{R}}$,
columns indexed by $\mathcal{K}$, and
entries defined by
$M_{t,k} = 1$ if the pair in $\mathcal{P}_{\mathrm{R}}$ indexed by $t$ is a member of the element of $\mathcal{K}$ indexed by $k$,
$M_{t,k} = 0$ otherwise.
Let $\bold{j}$ be the all-ones vector of dimension $|\mathcal{P}_{\mathrm{R}}|$.
If the design exists, then there must exist a $\{0,1\}$ vector $\bold{d}$ of dimension $|\mathcal{K}|$
that satisfies $M \bold{d} = \bold{j}$.

\subsection*{Algorithm C}~ 

Consider a $\delta$-regular graph with vertices $\{1, 2, \dots, (\delta + 1)^2\}$ and girth at least 5.
Construct $\mathcal{P}_{\mathrm{R}}$ as in Algorithm A and observe that
$|\mathcal{P}_{\mathrm{R}}| = \delta (\delta + 1)^2.$
Let
$$T = \{j: \{1,j\} \in \mathcal{P}_{\mathrm{R}}\},$$
and note that $|T| = 2 \delta$ by Lemma~\ref{lem-n-delta2-1}.
Partition $T$ into two sets $\{T_1$, $T_2\}$ of $\delta$ elements each.
For $i \in \{1,2\}$, check if $\{1\} \cup T_i$ is a valid block;
that is, check that the $\delta(\delta - 1)/2$ pairs which occur as subsets of $T_i$ are elements of $\mathcal{P}_{\mathrm{R}}$.
If two valid blocks cannot be found for any of the $\displaystyle \frac12 {2\delta \choose \delta}$ ways of partitioning $T$, then a design does not exist.

\subsection*{Algorithm D}~ 

Consider a 3-regular graph with vertices $\{1, 2, \dots, 22\}$ and girth at least 5.
Construct $\mathcal{P}_{\mathrm{R}}$ as in Algorithm A.

Let $t = 1$, let
$$T = \{j: \{t,j\} \in \mathcal{P}_{\mathrm{R}}\},$$
and note that $|T| = 12$.
Partition $T$ into four triples, $T_1$, $T_2$, $T_3$, $T_4$.
For each $i \in \{1,2,3,4\}$, check if $\{t\} \cup T_i$ is a valid block.
If four valid blocks cannot be found for any of the
$$\dfrac{1}{4!}{12 \choose 3}{9 \choose 3}{6 \choose 3} = 15400$$
ways of partitioning $T$,
then a design does not exist.

Whenever four valid blocks are found perform a second iteration of the process with
$t$ replaced by the smallest element of $\{2,3, \dots, 22\} \setminus T$
and the pairs contained in $T_i$, $i \in \{1,2,3,4\}$, eliminated from $\mathcal{P}_{\mathrm{R}}$.


\adfhide
{
\section{Declarations}

%
There are no conflicts of interest.
%
%
No funds, grants, or other support were received during the preparation of this manuscript.
%
%
We do not analyse or generate any datasets.
}

\end{document}